\documentclass{amsart}

\usepackage{amsfonts}
\usepackage{amssymb}
\usepackage{amsmath}
\usepackage{mathrsfs}

\newtheorem{theorem}{Theorem}[section]
\newtheorem{lemma}[theorem]{Lemma}
\newtheorem{corollary}[theorem]{Corollary}
\newtheorem{proposition}[theorem]{Proposition}

\theoremstyle{definition}
\newtheorem{definition}[theorem]{Definition}
\newtheorem{example}[theorem]{Example}

\newtheorem{remark}[theorem]{Remark}

\numberwithin{equation}{section}

\newcommand{\R}{\mathbb R}
\newcommand{\N}{\mathbb N}
\newcommand{\B}{\mathbb B}
\newcommand{\Sfer}{\mathbb S}

\newcommand{\X}{\mathbb X}

\newcommand{\dom}{{\rm dom}\, }
\newcommand{\Grph}{{\rm Gr}\,}
\newcommand{\Lev}{{\rm Lev}\,}

\newcommand{\nullv}{\mathbf{0}}
\newcommand{\inff}{\underline{\alpha}}

\newcommand{\clco}{{\rm cl\, conv}\, }

\newcommand{\front}{{\rm bd}\, }
\newcommand{\conv}{{\rm conv}\, }
\newcommand{\inte}{{\rm int}\, }
\newcommand{\weakstar}{weak${}^*\, $ }

\newcommand{\Haus}{{\rm Haus}}

\newcommand{\lev}[2]{{\rm lev}_{\le #1}\, {#2}}
\newcommand{\ball}[2]{{\rm B}\left(#1, #2\right)}
\newcommand{\dist}[2]{{\rm dist}\left(#1,#2\right)}
\newcommand{\stsl}[2]{|\nabla #1|(#2)}

\newcommand{\Argmin}[1]{{\rm Argmin}(#1)}
\newcommand{\Fix}[1]{{\rm Fix}(#1)}
\newcommand{\shar}[1]{{\rm sha}(#1)}
\newcommand{\wshar}[1]{{\rm wsha}(#1)}

\newcommand{\Hadlowder}[1]{{#1}^\downarrow_\mathcal{H}}
\newcommand{\Hadupder}[1]{{#1}^\uparrow_\mathcal{H}}

\newcommand{\GDer}{\nabla}
\newcommand{\FDer}{\hat{\rm D}}

\newcommand{\CompConv}{\mathcal{K}}
\newcommand{\Exhaust}{\underline{\mathcal E}}
\newcommand{\SLin}{\mathcal{S}}
\newcommand{\qdsub}{\underline{\partial}}
\newcommand{\qdsup}{\overline{\partial}}
\newcommand{\qdif}{\mathcal{D}}
\newcommand{\Clasubd}{\partial^\circ}

\newcommand{\suppf}[2]{\varsigma(#1, #2)}
\newcommand{\Demdif}{{\ \underline\vartriangle\ }}
\newcommand{\Demcoqd}{{\qdif}^\Demdif\hskip-0.05cm}

\newcommand{\Exhaustder}[2]{{\mathcal E}^\downarrow_\mathcal{H}(#1,#2)}

\begin{document}

\title{Some dual conditions for global weak sharp minimality of nonconvex functions}

\author{A. Uderzo}

\address{}

\address{A. Uderzo: Dept. of Mathematics and Applications, Universit\`a di
Milano-Bicocca, Via Cozzi, 53 - 20125 Milano, Italy, e-mail: {\tt amos.uderzo@unimib.it}}



\date{\today}



\begin{abstract}
Weak sharp minimality is a notion emerged in optimization, whose
utility is largeley recognized in the convergence analysis of
algorithms for solving extremum problems as well as in the
study of the perturbation behaviour of such problems. In the
present paper some dual constructions of nonsmooth analysis,
mainly related to quasidifferential calculus and its recent
developments, are employed in formulating sufficient conditions
for global weak sharp minimality. They extend to nonconvex functions
a condition, which is  known to be valid in the convex case. A
feature distinguishing the results here proposed is that they
avoid to assume the Asplund property on the underlying space.
\end{abstract}

\maketitle


\section{Introduction and motivations}

Given a function $f:X\longrightarrow\R\cup\{+\infty\}$, consider the
minimization problem
$$
    \min_{x\in X} f(x).   \leqno(\mathcal{P})
$$
Besides the fundamental question on the existence of solutions to
$(\mathcal{P})$, i.e. minimizers, an issue of interest is also to
single out special behaviours of $f$ able to describe `how' minima
(or, in the global case, the minimum) are attained. Such an issue,
which reveals to be crucial for the convergence analysis of
algorithms designed to solve problem $(\mathcal{P})$, led
to define the notion of weak sharp minima, which is considered
here in its global formulation.

In what follows $(X,d)$ denotes a metric space. Given $\alpha\in\R
\cup\{+\infty\}$, by $\lev{\alpha}{f}=\{x\in X:\ f(x)\le\alpha\}$
the $\alpha$ sublevel set of $f$ is indicated. The distance of point
$x\in X$ from a set $A\subseteq X$ is denoted by
$\dist{x}{A}=\displaystyle\inf_{a\in A}d(x,a)$, with
the convention $\dist{x}{\varnothing}=+\infty$, for every $x\in X$.
Set for convenience-sake $\displaystyle\inf_{x\in X}f(x)=\inff$.
Throughout the paper it will be assumed that $f$ is bounded from
below on $X$, that is $\inff>-\infty$, unless otherwise stated.
Moreover, since the issue
under investigation loses its interest if $f\equiv+\infty$, $f$ will
be assumed also to be proper, i.e. $\dom(f)\ne\varnothing$.
$\Argmin{f}=\lev{\inff}{f}$ represents the set of all global solutions
to problem $(\mathcal{P})$, if any. The acronym l.s.c. (respectively,
u.s.c.) abbreviates lower (respectively, upper) semicontinuous, with
reference to both functions and set-valued maps.

\begin{definition}      \label{def:gwsha}
A function  $f:X\longrightarrow\R\cup\{+\infty\}$ defined on a metric
space is said to have {\em global weak sharp minimizers} if 
$\Argmin{f}\ne\varnothing$ and there exists $\sigma>0$ such that
\begin{equation}    \label{in:defwsharmin}
    \sigma\cdot\dist{x}{\Argmin{f}}\le f(x)-\inff,\quad\forall x\in X.
\end{equation}
The supremum over all values $\sigma$ satisfying inequality
(\ref{in:defwsharmin}) will be called {\em modulus of global weak
sharpness} of $f$ and denoted by $\wshar{f}$.

If, in particular, $\Argmin{f}$ reduces to a singleton $\{\bar x\}$,
so that inequality (\ref{in:defwsharmin}) becomes
\begin{equation}   \label{in:defsharmin}
   \sigma d(x,\bar x)\le f(x)-\inff,\quad\forall x\in X,
\end{equation}
such an element $\bar x$ is called {\em sharp minimizer} of $f$ and
the related sharpness modulus is denoted by $\shar{f}$.
\end{definition}

For historical remarks on the appearance of weak sharpness
in optimization the reader is referred to \cite{BurDen02}.
As a comment to Definition \ref{def:gwsha} it is worth noting
that, in a normed space setting, if $f$ admits a sharp minimizer,
then the growth condition (\ref{in:defsharmin}) immediately
implies that $f$ is a coercive function, in the sense that
$$
    \lim_{\|x\|\to\infty} f(x)=+\infty.
$$
Nevertheless, the notion of weak sharp minimality is completely
independent of coercivity. In fact, a function with global
weak sharp minimizers may happen to have sublevel set
$\Argmin{f}=\lev{\inff}{f}$ not bounded, thereby failing to be
coercive. On the other hand, a function such as $x\mapsto
\|x\|^2$ is trivially coercive but it admits no global weak sharp
minimizers. In other words, the nature of these two properties,
although both related to global minimality, appears to be totally
different.

Another aspect of Definition \ref{def:gwsha} to be remarked
is that global weak sharp minimality is a property of function
$f$ itself, not a property qualifying some of its global minimizers.
This makes global weak sharp minimality different from its
local counterpart. Since, whenever it exists, the global minimum
of a function is unique, in the global case it makes no sense
to speak of `weak sharp minim\hskip-0.01cm{\em a}'. For this
reason, in formulating Definition \ref{def:gwsha}, the term
`weak sharp minimizers' has been preferred to the one largely
employed in the existent literature, namely `weak sharp minima'.

Further features of weak sharp minimizers come up when $X$ has
some more structure. Let us mention in this concern that, if $X$
is a Banach space and $f$ is a convex function not constant on
$X$, then the occurence of global weak sharp minimizers is
incompatible with the differentiability of $f$
(see, for instance, Proposition 2.4 in \cite{NgZhe03}).

In what follows several  exemplary contexts from optimization
and variational analysis are illustrated, in which the concept
of weak sharp minimizer naturally occurs. The resulting insight
should enlighten the crucial role played by the notion under
consideration in the mentioned disciplines.

\begin{example}
A function $f:X\longrightarrow\R\cup\{+\infty\}$ defined on a metric
space is said to have a {\em global error bound} if there exists
$\zeta>0$ such that
$$
    \dist{x}{\lev{\alpha}{f}}\le\zeta[f(x)-\alpha]_+,\quad\forall
    \alpha\in\R,\ \forall x\in\ X,
$$
where, given any real $r$, it is $[r]_+=\max\{r,0\}$. Thus, the property
of having global weak sharp minimizers amounts to the existence
of a global error bound in the particular case in which $\alpha=
\inff$, with $\zeta\ge\wshar{f}^{-1}$.
Notice that, with the convention here adopted $\dist{x}
{\varnothing}=+\infty$, the existence of a global error bound
for a function bounded from below entails that $\Argmin{f}\ne
\varnothing$.
\end{example}

\begin{example} Let $F:X\longrightarrow 2^X$ a set-valued map
defined and taking closed and bounded values on the same
metric space $(X,d)$. Suppose that $F$ is a contraction
of $X$, i.e. there exists $\kappa\in [0,1)$ such that
$$
    \Haus(F(x_1),F(x_2))\le\kappa d(x_1,x_2),\quad\forall
    x_1,\, x_2\in X,
$$
where $\Haus(A,B)$ indicates the Hausdorff distance of
the closed and bounded sets $A$ and $B$.
If $X$ is metrically complete the displacement function
associated to $F$, namely function $\phi_F:X\longrightarrow
[0,+\infty)$ defined as
$$
    \phi_F(x)=\dist{x}{F(x)},
$$
turns out to admit global weak sharp minimizers. Actually, they
coincide with the fixed points of $F$. Indeed, if denoting by
$\Fix{F}$ the set of all fixed points of $F$, it is known from the
Covitz-Nadler theorem on fixed points that $\Fix{F}\ne\varnothing$
and, for every $\kappa'\in (\kappa,1)$, it results in
\begin{equation}
    \dist{x}{\Fix{F}}\le\frac{\dist{x}{F(x)}}{1-\kappa'},\quad\forall x\in X,
\end{equation}
(see \cite{CovNad70,Nadl69}).
Thus, one deduces that $\wshar{\phi_F}\ge 1-\kappa$.
Notice that, generally speaking, function $\phi_F$ fails to admit a
global sharp minimizer. Nevertheless, since $\Fix{F}$ reduces to a
singleton whenever $F$ is single-valued on $X$ according to the
Banach-Caccioppoli contraction principle, then in the
latter case $\phi_F$ admits a global sharp minimizer and $\shar{\phi_F}
\ge 1-\kappa$.
\end{example}

\begin{example}
According to \cite{KlaKum02} a set-valued map $F:Y\longrightarrow 2^X$ between
metric spaces is said to be {\em calm} at $(y_0,x_0)\in\Grph F$ if
there exist $r>0$ and $\gamma>0$ such that
$$
    \dist{x}{F(y_0)}\le\gamma d(y,y_0),\quad\forall x\in F(y)\cap
   \ball{x_0}{r},\ \forall y\in\ball{y_0}{r},
$$
where $\ball{x}{r}=\lev{r}{d(\cdot,x)}$.
Let  $f:X\longrightarrow\R\cup\{+\infty\}$ be a given function.
Take $Y=\R$, $y_0=\inff$ and consider the set-valued map $\Lev f:\R
\longrightarrow 2^X$ associated to $f$ as follows
$$
    \Lev f(\alpha)=\lev{\alpha}{f}.
$$
Assume that $\Argmin{f}\ne\varnothing$ and that $\Lev f$ is calm
at each pair $(\inff,x)$, where $x\in\Argmin{f}$, and with $r=+\infty$
and a uniform constant $\gamma$. Then, being
$$
    \dist{x}{\Argmin{f}}=\dist{x}{\Lev f(\inff)}\le\gamma |f(x)-\inff|=
    \gamma(f(x)-\inff),\quad\forall x\in\Lev f(f(x))=X,
$$
one sees that under the above global calmness condition
function $f$ has global weak sharp minimizers.
\end{example}

An useful achievement of modern variational analysis is that the
calmness property of a given map captures a Lipschitz behaviour
that generalizes the metric regularity of its inverse, known as
metric subregularity. This fact leads to the next context in which
weak sharp minimality arises.

\begin{example}
Let $f:X\longrightarrow\R\cup\{+\infty\}$ be a given function
defined on a metric space. According to \cite{Ioff79}, $f$ is said to be
{\em metrically subregular} at $x_0\in X$ for $\alpha_0\in\R$ if
there exist $r>0$ and $\zeta>0$ such that
$$
    \dist{x}{f^{-1}(\alpha_0)}\le\zeta |f(x)-\alpha_0|,\quad
   \forall x\in\ball{x_0}{r}.
$$
From the above inequality, one readily sees that any function $f$,
which is metrically subregular at each point of $X$ for $\inff$,
with the same constant $\zeta$, admits global weak sharp minimizers.
\end{example}

\begin{example}
A given function  $f:X\longrightarrow\R\cup\{+\infty\}$ defined on a
topological space $X$ is said to be {\em Tikhonov well-posed} if
$\Argmin{f}=\{\bar x\}$ and for every sequence $(x_n)_{n\in\N}$ in
$X$, such that $f(x_n)\to\inff$ as $n\to\infty$ (minimizing sequence),
one has $x_n\to\bar x$.
It is clear that, in a metric space setting, whenever $f$ admits
a global sharp minimizer it is also Tikhonov well-posed. The
converse is evidently false, in general. If $f$ has global weak sharp minimizers
without a global sharp minimizer, from inequality (\ref{in:defwsharmin})
one easily deduces that any minimizing sequence is metrically
attracted by $\Argmin{f}$. Thus, global weak sharp minimality leads to
a set valued-like generalization of the Tikhonov well-posedness.
\end{example}

In consideration of its theoretical and computational relevance, weak
sharp minimality has been the subject of manifold investigations
within optimization and variational analysis (see, for instance,
\cite{Bedn07,BurDen02,BurDen05,BurDen09,BurFer93,NgZhe03,StuWar99,
Zali01}).
A topic considered in several of them is the problem of finding methods to
detect the occurence of weak sharp minimality. The present paper
intends to focus on this topic. Following a recognized line of reasearch,
the goal of the analysis here proposed is to extend a known condition
valid for convex functions beyond the realm of convexity. Since, as
already mentioned, weak sharp minimality is incompatible with
differentiability, this task is pursued by making use of nonsmooth
analysis tools. In particular, sufficient conditions for global weak
sharp minimality are presented in the case of problems $(\mathcal{P})$
with objective function quasidifferentiable or with generalized
derivatives admitting lower exhausters. In both cases, the key role
is played by respective nondegeneracy conditions involving
subdifferential-like dual constructions. A feature of the present
analysis is that the resulting conditions are achieved as a consequence
of a general analysis conducted in a metric space setting.

The material exposed in rest of the paper is arranged as follows.
In Section \ref{Sect1} a characterization and a sufficient condition
for global weak sharp minimality are established and discussed in
a metric space setting.
In Section \ref{Sect2} some material from nonsmooth analysis,
which is needed in order to formulate nondegeneracy conditions,
is briefly recalled and some related ancillary result is proved.
Section \ref{Sect3} is reserved to present and comment the main
results of the paper. They consider both unconstrained and variously
constrained extremum problems.


\section{Global weak sharp minimality in metric spaces}   \label{Sect1}

Even though metric space is in structure too poor for certain
applications, nevertheless it should be the proper environment
where to analyze the notion of weak sharp minimality `juxta propria
principia', in consideration of the purely metric nature of its
definition. In fact, all dual conditions presented in Section
\ref{Sect3} will be derived from basic results established
in the present section.
Let us start with a characterization for $f$ to have
global weak sharp minimizers, which relies on the behaviour of
sublevel sets of $f$.

\begin{theorem}     \label{thm:wsharchalev}
Let $(X,d)$ be a complete metric space. A function $f:X\longrightarrow
\R\cup\{+\infty\}$ l.s.c. on $X$ admits global  weak sharp minimizers
with modulus  $\wshar{f}\ge\tau>0$ iff
\begin{equation}      \label{in:wshachar}
    \tau\cdot\sup_{\alpha\in (\inff,+\infty)}\dist{x}{\lev{\alpha}{f}}\le
    f(x)-\inff,\quad\forall x\in X.
\end{equation}
\end{theorem}

\begin{proof}
Suppose first that $f$ does admit global weak sharp minimizers,
having modulus  $\wshar{f}$. Then, since it is
$$
   \varnothing\ne\Argmin{f}\subseteq\lev{\alpha}{f},\quad
   \forall \alpha\in (\inff,+\infty),
$$
consequently, for every $\sigma\in (0,\wshar{f})$ and $ \alpha
\in (\inff,+\infty)$, it holds
$$
   \sigma\cdot\dist{x}{\lev{\alpha}{f}}\le\sigma\cdot
   \dist{x}{\Argmin{f}}\le f(x)-\inff,\quad\forall x\in X,
$$
whence inequality (\ref{in:wshachar}) follows at once.

Conversely, suppose condition (\ref{in:wshachar}) to hold true.
Fix an arbitrary $x_0\in X\backslash\Argmin{f}$. Without loss of generality it is
possible to assume that $x_0\in\dom(f)$. Indeed, in the case $x_0
\not\in\dom(f)$ inequality (\ref{in:defwsharmin}) would be
automatically true (remember that $\dom(f)\ne\varnothing$).
Take an arbitrary $\sigma\in (0,\tau)$. Since it is
$f(x_0)\le\inff+(f(x_0)-\inff)$, then by virtue of the Ekeland variational
principle, corresponding to the value
$$
    \lambda=\frac{f(x_0)-\inff}{\sigma},
$$
there exists $x_\lambda\in X$ with the properties:
\begin{equation}      \label{in:charEVP1}
      f(x_\lambda)\le f(x_0),
\end{equation}
\begin{equation}      \label{in:charEVP2}
    d(x_\lambda,x_0)\le\lambda,
\end{equation}
\begin{equation}      \label{in:charEVP3}
     f(x_\lambda)<f(x)+\sigma d(x,x_\lambda),\quad\forall
     x\in X\backslash\{x_\lambda\}.
\end{equation}
Inequality (\ref{in:charEVP1}) says that $x_\lambda\in\dom(f)$.
Let us show that such an $x_\lambda$ is a global weak sharp
minimizer of $f$. Indeed, assume ab absurdo that $x_\lambda
\not\in\Argmin{f}$, so $\inff<f(x_\lambda)$. Then, condition
(\ref{in:wshachar}) being valid by hypothesis, from the fact that
$$
    \dist{x_\lambda}{\lev{\alpha}{f}}>0,\quad\forall\alpha\in
    (\inff,f(x_\lambda)),
$$
one obtains
$$
    \sigma\cdot\sup_{\alpha\in (\inff,f(x_\lambda))}\dist{x_\lambda}{\lev{\alpha}{f}}
    <\tau\cdot\sup_{\alpha\in (\inff,+\infty)}\dist{x_\lambda}{\lev{\alpha}{f}}
    \le f(x_\lambda)-\inff.
$$
This means that it is possible to find $\epsilon>0$ such that
$$
    \sigma\cdot\sup_{\alpha\in (\inff,f(x_\lambda))}\dist{x_\lambda}{\lev{\alpha}{f}}
    <f(x_\lambda)-\inff-\epsilon.
$$
In particular, as it is $\inff<\inff+\epsilon<f(x_\lambda)$, it holds
$$
    \sigma\cdot\dist{x_\lambda}{\lev{\inff+\epsilon}{f}}<f(x_\lambda)
   -\inff-\epsilon.
$$
From the last inequality it follows that there exists $x_\epsilon\in
\lev{\inff+\epsilon}{f}$ such that
$$
    \sigma d(x_\lambda,x_\epsilon)<f(x_\lambda)-\inff-\epsilon.
$$
Notice that $x_\epsilon\in X\backslash\{x_\lambda\}$, because
$x_\lambda$ does not belong to $\lev{\inff+\epsilon}{f}$.
Thus, by recalling inequality (\ref{in:charEVP3}), one finds
$$
   f(x_\lambda)<f(x_\epsilon)+\sigma d(x_\epsilon,x_\lambda)<
   f(x_\epsilon)+f(x_\lambda)-\inff-\epsilon\le f(x_\lambda),
$$
which clearly leads to an absurdum. This argument hence shows 
that $x_\lambda\in\Argmin{f}\ne\varnothing$. Again, by virtue of
inequality (\ref{in:charEVP2}), by recalling the chosen value of
$\lambda$ one obtains
$$
    \dist{x_0}{\Argmin{f}}\le d(x_0,x_\lambda)\le\frac{f(x_0)-\inff}
    {\sigma},
$$
whence
\begin{equation}    \label{in:wsharestim}
    \sigma\cdot\dist{x_0}{\Argmin{f}}\le  f(x_0)-\inff.
\end{equation}
Since the validity of inequality (\ref{in:wsharestim}) has been proved
for every $\sigma\in (0,\tau)$, one can deduce that
$$
    \tau\cdot\dist{x_0}{\Argmin{f}}\le  f(x_0)-\inff.
$$
By arbitrariness of $x_0$, all requirements of Definition \ref{def:gwsha}
appear now to be fulfilled. Therefore the proof is complete.
\end{proof}

\begin{remark}
It is helpful to mention that, since for every $x\in X$ function
$\alpha\mapsto\dist{x}{\lev{\alpha}{f}}$ is monotone decreasing
on $[\inff,+\infty)$, then condition (\ref{in:wshachar}) can be
equivalently rewritten as
$$
    \tau\cdot\lim_{\alpha\to\inff^+}\dist{x}{\lev{\alpha}{f}}\le
    f(x)-\inff,\quad\forall x\in X.
$$
In fact, Theorem \ref{thm:wsharchalev} has been devised as a
modification of a similar result presented in \cite{NgZhe03},
where metric completeness has been exploited through
the shrinking ball property, instead of through the Ekeland's principle.
\end{remark}

A basic sufficient condition for global weak sharp minimality
can be formulated in terms of strong slope. This is a variational
analysis tool originally proposed in \cite{DeMaTo80} for quite
different purposes, whose utility has been demonstrated in several
circumstances (see, for instance, \cite{AzCoLu02,Ioff00}).
Given a function
$f:X\longrightarrow\R\cup\{\pm\infty\}$ defined on a metric space,
by {\it strong slope} of $f$ at $\bar x\in\dom(f)$ the real-extended
value
$$
   \stsl{f}{\bar x}=\left\{ \begin{array}{ll}
                       0, & \qquad\text{if $\bar x$ is a local minimizer for $f$},  \\
                       \displaystyle\limsup_{x\to\bar x} 
                       \frac{f(\bar x)-f(x)}{d(x,\bar x)}, & 
                       \qquad\text{otherwise,} 
                 \end{array}
   \right.
$$
is meant. In the case $\bar x\not\in\dom(f)$, set $\stsl{f}{\bar x}=+\infty$.

\begin{example}      \label{ex:stsldifconv}
The above recalled tool takes a form more appealing from the
computational viewpoint when $X$ and $f$ possess more structure.
For instance, if $X$ is a normed vector space, then for any
function $f$ Fr\'echet differentiable at $x\in X$, with Fr\'echet
derivative $\FDer f(x)$, it holds
$$
     \stsl{f}{x}=\|\FDer f(x)\|,
$$
(what motivates the notation in use). If, in the same space,
$f$ is convex and subdifferentiable at $x$, it holds
$$
     \stsl{f}{x}=\dist{\nullv^*}{\partial f(x)},
$$
where $\partial f(x)$ denotes the subdifferential of $f$ at $x$
in the sense of convex analysis and $\nullv^*$ stands for the
null vector of the dual space (see Section \ref{Sect2}).
\end{example}

\begin{theorem}     \label{thm:wsharsufcondstsl}
Let $(X,d)$ be a complete metric space and let $f:X\longrightarrow
\R\cup\{+\infty\}$ be l.s.c. on $X$. If there exists $\tau>0$ such that
\begin{equation}    \label{inc:stslwshacon}
    \stsl{f}{X\backslash\Argmin{f}}\subseteq [\tau,+\infty),
\end{equation}
then $f$ admits global weak sharp minimizers and $\wshar{f}\ge\tau$.
\end{theorem}

\begin{proof}
Observe first that if $\Argmin{f}=X$ the thesis becomes trivial.
So, fix arbitrarily $x_0$ and $\sigma$ in $(X\backslash\Argmin{f})
\cap\dom(f)$ and $(0,\tau)$, respectively. In the case $x_0\not\in
\dom(f)$ there is nothing to be proved.
According to the Ekeland variational principle, corresponding to
$\lambda=(f(x_0)-\inff)/\sigma$ there exists $x_\lambda\in\dom(f)$
such that
\begin{equation}    \label{in:EVP2bis}
   d(x_\lambda,x_0)\le\lambda
\end{equation}
and
$$
   f(x_\lambda)<f(x)+\sigma d(x,x_\lambda),\quad\forall x\in X
   \backslash\{x_\lambda\}.
$$
The last inequality implies that for every $\delta>0$ it results in
$$
  \sup_{x\in\ball{x_\lambda}{\delta}\backslash\{x_\lambda\}}
  \frac{f(x_\lambda)-f(x)}{d(x,x_\lambda)}\le
  \sup_{x\in X\backslash\{x_\lambda\}}
  \frac{f(x_\lambda)-f(x)}{d(x,x_\lambda)}\le\sigma,
$$
wherefrom it follows
\begin{equation}     \label{in:smallstsl}
     \stsl{f}{x_\lambda}\le\sigma.
\end{equation}
Now observe that it must be $x_\lambda\in\Argmin{f}$. Otherwise,
if it were $x_\lambda\in X\backslash\Argmin{f}$, inequality
(\ref{in:smallstsl}) would contradict condition (\ref{inc:stslwshacon}).
From inequality (\ref{in:EVP2bis}) one obtains
$$
    \sigma\dist{x_0}{\Argmin{f}}\le f(x_0)-\inff.
$$
The arbitrariness of $x_0$ and $\sigma$ in their respective domains
allows one to complete the proof.
\end{proof}

The next example shows that, in contrast to the previous result,
Theorem \ref{thm:wsharsufcondstsl} is far removed from being
a characterization of global weak sharp minimality.

\begin{example}       \label{ex:contexsufcond}
Let $X=\R$ be equipped with its usual (Euclidean)
metric structure. Consider function $f:\R\longrightarrow\R$
defined by
$$
    f(x)=\left\{\begin{array}{ll}
                       0, & \qquad\text{if}\ x\in (-\infty,0],  \\
                       \displaystyle x+\frac{1}{x+1}, & 
                       \qquad\text{otherwise.} 
                 \end{array}
   \right.
$$
Clearly, $f$ is l.s.c. and bounded from below on $\R$. Here $\inff=0$ and
$\Argmin{f}=(-\infty,0]$. Thus, since it results in
$$
    \dist{x}{\Argmin{f}}=[x]_+\le x+\frac{1}{x+1},\quad\forall x\in
   (0,+\infty),
$$
one deduces that $f$ admits global weak sharp minimizers, with
$\wshar{f}\ge 1$. Nonetheless, being $f$ differentiable on $\R
\backslash\{0\}$, by taking into account Example \ref{ex:stsldifconv},
one finds
$$
   \stsl{f}{x}=1-\frac{1}{(x+1)^2},\quad\forall x\in (0,+\infty),
$$
what makes condition (\ref{inc:stslwshacon}) evidently violated.
\end{example}

\begin{remark}
As, according to the definition of strong slope, it is
$\stsl{f}{\Argmin{f}}=\{0\}$, condition (\ref{inc:stslwshacon})
entails that, if $f$ is not constant, function $|\nabla f|:X
\longrightarrow [0,+\infty]$ can not be u.s.c. (and hence continuous)
at each point $\bar x\in\Argmin{f}$, which is an accumulation point
of $X\backslash\Argmin{f}$. Indeed, for any sequence $(x_n)_{n\in\N}$,
with $x_n\in X\backslash\Argmin{f}$ and $x_n\to\bar x$ as $n\to\infty$,
one finds
$$
   \limsup_{n\to\infty}\stsl{f}{x_n}\ge\liminf_{n\to\infty}\stsl{f}{x_n}
   \ge\tau>0=\stsl{f}{\bar x}.
$$
\end{remark}

Theorem \ref{thm:wsharsufcondstsl} enables one to easily derive
a global error bound for the solution set of an inequality/equality
system. Given $g:X\longrightarrow\R\cup\{\pm\infty\}$ and
$h:X\longrightarrow\R\cup\{\pm\infty\}$, for any $\alpha,\,\beta\in\R$
define
$$
    \Omega_{\alpha,\beta}=\lev{\alpha}{g}\cap h^{-1}(\beta).
$$

\begin{corollary}      \label{cor:stslerbo}
Let $(X,d)$ be a complete metric space, let $g:X\longrightarrow
\R\cup\{\pm\infty\}$ be l.s.c. on $X$, let $h:X\longrightarrow\R$
be continuous on $X$, both not necessarily bounded from below,
and let $\alpha,\,\beta\in\R$. If $\Omega_{\alpha,\beta}\ne\varnothing$
and there exists $\tau>0$ such that
\begin{equation*}  
    \stsl{([g-\alpha]_++|h-\beta|)}{X\backslash\Omega_{\alpha,\beta}}
    \subseteq [\tau,+\infty),
\end{equation*}
then the following error bound holds
$$
    \dist{x}{\Omega_{\alpha,\beta}}\le\tau^{-1} \left([g(x)-\alpha]_+
    +|h(x)-\beta|\right),\quad\forall x\in X.
$$
\end{corollary}

\begin{proof}
Set $f=[g-\alpha]_++|h-\beta|$. Since by hypothesis
$\Omega_{\alpha,\beta}\ne\varnothing$, then it is $\displaystyle
\inf_{x\in X}f(x)=0$. Such an infimum is actually attained and one has
$\Omega_{\alpha,\beta}=\Argmin{f}$. By definition, under the
assumptions made, function $f$ is l.s.c. and obviously bounded from below.
Thus, one is in a position to apply Theorem \ref{thm:wsharsufcondstsl}.
\end{proof}


\section{Selected elements of Nonsmooth Analysis}    \label{Sect2}

Throughtout the current section, $(\X,\|\cdot\|)$ denotes a real
Banach space. The (topological) dual space of $\X$ is marked
by $\X^*$, with $\X^*$ and $\X$ being paired in duality
by the bilinear form $\langle\,\cdot,\cdot\rangle:\X^*\times\X
\longrightarrow\R$. The null vector of $\X$ is indicated by $\nullv$,
while the null functional by $\nullv^*$.
$\B^*$ denotes the unit ball of $\X^*$, while $\Sfer$
denotes the unit sphere in $\X$. The convention $A+\varnothing=
\varnothing$ is adopted for any subset $A$ of a vector space.
The commutative semigroup of all (Lipschitz) continuous sublinear
functions defined on $\X$ is denoted by $\SLin(\X)$. A starting point for entering
the subsequent dual constructions is the semigroup isomorphism
between $\SLin(\X)$ and the semigroup of all nonempty convex and 
\weakstar compact subsets of $\X^*$, denoted by $\CompConv(\X^*)$.
Such isomorphism is known as {\em Minkowski-H\"ormander duality}
and is represented by the subdifferential map $\partial:\SLin(\X)
\longrightarrow\CompConv(\X^*)$ as follows
$$
    \partial(h)=\partial h(\nullv),\quad\ h\in \SLin(\X).
$$
Clearly, $\partial^{-1}:\CompConv
(\X^*)\longrightarrow\SLin(\X)$ can be defined
via the support function to a given subset, namely
$$
    \partial^{-1}(A)=\suppf{\cdot}{A},\quad A\in\CompConv(\X^*),
$$
where $\displaystyle\suppf{x}{A}=\max_{x^*\in A}\langle x^*,x\rangle$.

\subsection{Quasidifferentiable functions and Demyanov difference}

Given a function $f:\X\longrightarrow\R\cup\{\pm\infty\}$
and $\bar x\in\dom(f)$, let $f'(\bar x;v)$ denote the directional
derivative of $f$ at $\bar x$ in the direction $v\in\X$. According
to \cite{DemRub80}, function $f$ is said to be {\em quasidifferentiable}
(for short, {\em q.d.}) at $\bar x$ if it admits directional
derivative at $\bar x$ in all directions and there exist two
elements $\underline{f},\,\overline{f}\in\SLin(\X)$ such that
$$
   f'(\bar x;v)=\underline{f}(v)-\overline{f}(v),\quad\forall v\in\X.
$$
In the light of the Minkowski-H\"ormander duality this amounts 
to say that the following dual representation is valid
\begin{equation}    \label{eq:defquasidif}
   f'(\bar x;v)=\suppf{v}{\partial\underline{f}(\nullv)}-
      \suppf{v}{\partial\overline{f}(\nullv)}, \quad\forall v\in\X.
\end{equation}
Clearly, the representation in (\ref{eq:defquasidif}) of $f'(\bar x;\cdot)$,
as well as the previous one in terms of $\SLin(\X)$, is by no means unique.
This is not a serious drawback, because every pair of elements of
$\CompConv(\X^*)$ representing $f'(\bar x;\cdot)$ belongs to the same
class with respect to an equivalence relation $\sim$ defined on
$\CompConv(\X^*)\times\CompConv(\X^*)$, according
to which $(A,B)\sim (C,D)$ if $A+D=B+C$. The $\sim$-equivalence
class containing the pair $(\partial\underline{f}(\nullv),-\partial
\overline{f}(\nullv))$ is called {\em quasidifferential}  of $f$ at $\bar x$ and
will be denoted in further constructions by $\qdif f(\bar x)$.
Any pair in the class $\qdif f(\bar x)$ will be henceforth indicated
by $(\qdsub f(\bar x),-\qdsup f(\bar x))$, so
\begin{equation*}
   f'(\bar x;v)=\suppf{v}{\qdsub f(\bar x)}
  -\suppf{v}{-\qdsup f(\bar x)},\quad\forall v\in\X,
   \qquad\hbox{and}\qquad \qdif f(\bar x)=
  [\qdsub f(\bar x),-\qdsup f(\bar x)]_\sim.
\end{equation*}

In the early 80-ies a complete calculus for quasidifferentiable functions
has been developed, which finds a geometric counterpart in the
calculus for $\sim$-equivalence classes of pairs in $\CompConv(\X^*)
\times\CompConv(\X^*)$ (see \cite{DemRub86,DemRub95,Rubi92}). 
This, along with a notable computational tractability of the resulting
constructions, made such approach a recognized and successful
subject within nonsmooth analysis.

The next step towards the setting of analysis tools in use in the
subsequent section requires the introduction of a difference
operation in $\CompConv(\X^*)$. As illustrated in several works
(see, for instance, \cite{DemRub95,Rubi92,RubAkh92,RubVla00}),
such a task can be accomplished following different approaches.
For the purposes of the present investigations, the following notion
seems to be adequate.

\begin{definition}
The inner operation $\Demdif:\CompConv(\X^*)\times
\CompConv(\X^*)\longrightarrow\CompConv(\X^*)$,
defined by
\begin{equation}   \label{def:Demdiff}
   A\Demdif B=\Clasubd(\suppf{\cdot}{A}-\suppf{\cdot}{B})(\nullv),
\end{equation}
where $\Clasubd$ denotes the Clarke subdifferential operator
(see \cite{Clar83,DemRub95}),
is called {\em Demyanov difference} of $A$ and $B$.
\end{definition}

\begin{remark}
Notice that, since function $\suppf{\cdot}{A}-\suppf{\cdot}{B}$ is
Lipschitz continuous, the expression in (\ref{def:Demdiff}) actually
makes sense. In other words, $A\Demdif B$ never gives $\varnothing$.
It is to be mentioned that the definition proposed above is not
the original one, as it was introduced in \cite{Demy80}.
The latter was formulated in the more particular setting of finite
dimensional Euclidean spaces. It is nonetheless relevant
to recall it in detail in as much as this should offer
insights into the potential of such tool. Fixed a point $x\in\R^n$
and an element $A\in\CompConv(\R^n)$, the set
$$
    \Phi_x(A)=\{a\in A:\langle a,x\rangle=\suppf{x}{A}\}
$$
is called {\em max-face} of $A$ generated by $x$. The dual representation
$$
    \suppf{x}{A}=\max_{a\in\partial\suppf{\cdot}{A}(x)}\langle a,x\rangle
$$
shows that $\Phi_x(A)=\partial\suppf{\cdot}{A}(x)$. Now, from the
differentiability theory of convex functions it is well known that
$\partial\suppf{\cdot}{A}(x)$ is a singleton iff function $\suppf{\cdot}{A}$
admits gradient $\GDer\suppf{\cdot}{A}(x)$ at $x$. Observe that
in such a circumstance the max-face $\Phi_x(A)$ reduces to what is
called an exposed point of $A$ generated by the hyperplane
$\langle\cdot\, ,x\rangle$. Then, according to the Rademacher's theorem,
setting
$$
   M_A=\{x\in\R^n: \Phi_x(A) \hbox{ is a singleton}\}
$$
and denoting by $\mu$ the Lebesgue measure on $\R^n$, it results in
$$
    \mu(\R^n\backslash M_A)=0,
$$
i.e. $M_A$ is a Lebesgue full-measure set.
Now recall that, given a locally Lipschitz function $\varphi:
\R^n\longrightarrow\R$, a point  $\bar x\in\R^n$ and a full-measure
set $M$, at the points of which $\varphi$ admits gradient, it holds
$$
   \Clasubd \varphi(\bar x)=\conv \{v\in\R^n:\ \exists (x_n)_{n\in\N},\ 
   x_n\in M,\ x_n\to\bar x,\ v=\lim_{n\to\infty}\GDer\varphi(x_n)\},
$$
where $\conv$ denotes the convex hull of a given set.
Consequently, letting $M_{A,B}$ be a full-measure subset of $\R^n$
where both functions $\suppf{\cdot}{A}$ and $\suppf{\cdot}{B}$ admits
gradient, it is possible to define
$$
    A\Demdif B=\clco \{\GDer\suppf{\cdot}{A}(x)-\GDer\suppf{\cdot}{B}(x):\
    x\in M_{A,B}\},
$$
where $\clco$ indicates convex closure.
This finite dimensional reading of $\Demdif$ allows one for the following
constructive view of $A\Demdif B$: such set turns out to consist of all differences
of points, respectively in $A$ and $B$, which are exposed by the
same hyperplane $\langle\cdot\, ,x\rangle$, with $x$ varying in
$M_{A,B}$.
\end{remark}

The below lemma collects those properties of the Demyanov difference
that will be exploited in the sequel.  Their proofs, as well as additional
material and further discussion on this topic, can be found in
\cite{DemRub95,Rubi92,RubAkh92,RubVla00}.

\begin{lemma}    \label{lem:demdif}
(1) For every $A,\, B,\, C,\, D\in\CompConv(\X^*)$, if
$(A,B)\sim (C,D)$ then $A\Demdif B=C\Demdif D$;

(2) for every $A\in \CompConv(\X^*)$, it holds $A\Demdif A=\{\nullv^*\}$;

(3)  for every $A,\, B\in\CompConv(\X^*)$, if $B\subseteq A$, then
$\nullv^*\in A\Demdif B$;

(4) for every $A,\, B,\, C,\, D\in\CompConv(\X^*)$, the inclusion
$(A+B)\Demdif (C+D)\subseteq (A\Demdif C)+(B\Demdif D)$
holds true;

(5) for every $A\in \CompConv(\X^*)$, it holds $A\Demdif\{\nullv^*\}
=A$;

(6) for every $A,\, B\in\CompConv(\X^*)$, it is $A\Demdif B\subseteq
A-B$.
\end{lemma}

Given a function $f:\X\longrightarrow\R$, suppose that $f$ is q.d.
at each point of $\X$, with $\qdif f(x)=[\qdsub f(x),-\qdsup f(x)]_\sim$.
By combining the quasidifferential pairs of $f$ with Demyanov
difference  the following generalized derivative construction
$\Demcoqd f:\X\longrightarrow\CompConv(\X^*)$, which
will be employed to establish a weak sharp minimality condition,
is obtained:
\begin{eqnarray*}
    \Demcoqd f(x)=\qdsub f(x)\Demdif (-\qdsup f(x)).
\end{eqnarray*}

The use of the above construction is not new in nonsmooth
analsysis (as an example of different employments, see
for instance \cite{Gao00,Uder07}).

\begin{remark}    \label{rem:demdif}
(1) Notice that, by virtue of Lemma \ref{lem:demdif}(1), map $\Demcoqd f$
does not depend on particular representations of  $\qdif f$.
Since, as one immediately checks, one finds $\Demcoqd h(\nullv)
=\partial(h)$ whenever $h\in\SLin(\X)$, construction $\Demcoqd$
can be regarded as an extension of the Minkowski-H\"ormander
duality. 

(2) In force of Lemma \ref{lem:demdif}(3) it is readily seen that,
whenever $\bar x\in\X$ is a local minimizer of $f$, it must be
$\nullv^*\in \Demcoqd f(\bar x)$.

(3) By employing Lemma \ref{lem:demdif}(4) and the well-known
sume rule for quasidifferentials (see, for instance, \cite{DemRub95}),
it is possible to prove
that, given two functions  $f:\X\longrightarrow\R$ and  $g:\X
\longrightarrow\R$, both q.d. at $x\in\X$, the following inclusion holds
$$
     \Demcoqd (f+g)(x)\subseteq \Demcoqd f(x)+\Demcoqd g(x).
$$

(4) By employing the calculus rule for quasidifferentials
of functions by scalars, it is possible to prove that, given a function $f:
\X\longrightarrow\R$ q.d. at $x\in\X$ and  a scalar $\alpha\ge 0$,
the following equality holds
$$
     \Demcoqd [\alpha f](x)\subseteq \alpha\Demcoqd f(x).
$$

(5) From Lemma \ref{lem:demdif}(6) it is evident that, if a
function $f:\X\longrightarrow\R$ is q.d. at $x$, with $\qdif f(x)=
[\qdsub f(x),-\qdsup f(x)]_\sim$, then it is always $\Demcoqd
f(x)\subseteq\qdsub f(x)+\qdsup f(x)$. Immediate examples
show that this inclusion can be strict. Notice that the right side
term of it depends on the pair in $\CompConv(\X^*)\times
\CompConv(\X^*)$ representing $\qdif f(x)$.
\end{remark}

\begin{example}    \label{ex:convf}
Let $f:\X\longrightarrow\R$ be a continuous convex function.
Since in this case $f$ is q.d. at each point $x\in\X$ with
$\qdif f(x)=[\partial f(x),\{\nullv^*\}]_\sim$, for such kind of
function owing to Lemma \ref{lem:demdif}(5) one obtains
$$
    \Demcoqd f(x)=\partial f(x)\Demdif\{\nullv^*\}=\partial
    f(x).
$$
When, in particular, $f$ happens to be G\^ateaux differentiable
at $x$, with G\^ateaux derivative $\GDer f(x)$, one has $\Demcoqd
f(x)=\{\GDer f(x)\}$. Therefore, if $\X$ is a separable Banach space,
map $\Demcoqd f$ is single valued on a $G_\delta$ dense subset
of $\X$, on account of Mazur theorem.
\end{example}

In view of the employment of construction $\Demcoqd$ in the analysis
of global weak sharp minimality through the condition established
in Theorem \ref{thm:wsharsufcondstsl}, the next estimate, already
obtained in \cite{Uder07} within a more general argument, is needed.

\begin{lemma}      \label{lem:stslDemcoqdestim}
Given a function $f:\X\longrightarrow\R\cup\{\pm\infty\}$ defined
on a normed vector space $(\X,\|\cdot\|)$, suppose that $f$ is q.d.
at $\bar x\in\dom(f)$. Then, the following inequality holds
\begin{equation}     \label{in:stslDemcoqdest}
     \stsl{f}{\bar x}\ge\dist{\nullv^*}{\Demcoqd f(\bar x)}.
\end{equation}
\end{lemma}

\begin{proof}
If $\bar x$ happens to be a local minimizer of $f$ then, as stated in Remark
\ref{rem:demdif}(2), one has $\nullv^*\in\Demcoqd f(\bar x)$.
In such event inequality (\ref{in:stslDemcoqdest}) is trivially satisfied.
Otherwise, fix an arbitrary $\epsilon>0$. Corresponding to it,
by definition of strong slope, there exists $\delta_\epsilon>0$ such that
$$
     \sup_{x\in\ball{\bar x}{\delta_\epsilon}\backslash\{\bar x\}}
     \frac{f(\bar x)-f(x)}{d(x,\bar x)}<\stsl{f}{\bar x}+\epsilon.
$$
This means that $\bar x$ is a local minimizer of function $f(\cdot)
+(\stsl{f}{\bar x}+\epsilon)\|\cdot-\bar x\|$. By applying again what
has been noted in Remark \ref{rem:demdif}(2), (3) and (4), one obtains
\begin{eqnarray*}
  \nullv^* &\in &\Demcoqd [f(\cdot)+(\stsl{f}{\bar x}+\epsilon)\|\cdot-\bar x\|]
  (\bar x)\subseteq \Demcoqd f(\bar x)+(\stsl{f}{\bar x}+\epsilon)
  \Demcoqd \|\cdot-\bar x\|(\bar x) \\
  &=& \Demcoqd f(\bar x)+(\stsl{f}{\bar x}+\epsilon)\B^*.
\end{eqnarray*}
Such an inclusion implies the existence of $v^*\in\Demcoqd f(\bar x)$
such that $\|v^*\|\le\stsl{f}{\bar x}+\epsilon$.  This allows one to deduce that
$$
   \dist{\nullv^*}{\Demcoqd f(\bar x)}\le\stsl{f}{\bar x}+\epsilon,
$$
whence the thesis follows by arbitrariness of $\epsilon$.
\end{proof}


\subsection{Lower exhausters of generalized derivatives}

Let $h:\X\longrightarrow\R\cup\{\pm\infty\}$ be a positively
homogeneous of degree one (henceforth, for short,
p.h.) function. The need of extending the Minkowski-H\"ormander
duality in such a way to dually represent classes of p.h. functions,
which are broader than $\SLin(\X)$, led to introduce the notion of
lower and upper families of exhausters (see \cite{Demy99}).
Accordingly, a family $\Exhaust(h)\subseteq\CompConv(\X^*)$
is said to be a {\em lower exhauster} of $h$ if
$$
    h(x)=\inf_{E\in\Exhaust(h)}\ \max_{x^*\in E}\langle x^*,
    x\rangle=\inf_{E\in\Exhaust(h)}\ \suppf{x}{E}, \quad\forall
    x\in\X.
$$
In nonsmooth optimization the p.h. functions to be dually represented
are often generalized derivatives. Among them, in the investigations
here exposed, the Hadamard directional derivatives will be employed.
Given a function $f:\X\longrightarrow\R\cup\{\pm\infty\}$ and
$\bar x\in\dom(f)$, denote by
$$
    \Hadlowder{f}(\bar x;v)=\liminf_{\scriptstyle v'\to v\atop
    \scriptstyle t\to 0^+}\frac{f(\bar x+tv')-f(\bar x)}{t}
    \quad\hbox{and}\quad
   \Hadupder{f}(\bar x;v)=\limsup_{\scriptstyle v'\to v\atop
    \scriptstyle t\to 0^+}\frac{f(\bar x+tv')-f(\bar x)}{t}
$$
respectively the {\em Hadamard lower derivative} and the {\em 
Hadamard upper derivative} of $f$ at $\bar x$, in the direction
$v\in\X$.
Notice that, if $f$ is directionally differentiable at $\bar x$
in all directions and it is locally Lipschitz around the same point,
then for every $v\in\X$ it is $\Hadlowder{f}(\bar x;v)= \Hadupder{f}
(\bar x;v)=f'(\bar x;v)$. In the subsequent section, when dealing
with a function $f$, whose Hadamard lower derivative at $\bar x$
admits lower exhauster, the shortened notation
$$
    \Exhaustder{f}{\bar x}=\Exhaust(\Hadlowder{f}(\bar x;\cdot)).
$$
will be used. In order to formulate a nondegeneracy condition
involving lower exhausters of Hadamard lower derivatives, it
is useful to define the `norm' of a family $\Exhaust(h)$ of elements
in $\CompConv(\X^*)$ as
$$
    \|\Exhaust(h)\|=\sup_{E\in\Exhaust(h)}\ \dist{\nullv^*}{E}.
$$


\section{Sufficient conditions in Banach spaces}     \label{Sect3}

\subsection{Weak sharp minimality conditions for q.d. extremum problems}

\begin{theorem}     \label{thm:qdwsha}
Let $(\X,\|\cdot\|)$ be a Banach space and let $f:\X\longrightarrow\R$
be a function l.s.c. on $\X$. Suppose that $f$ is q.d. at each point
of $\X$ and
\begin{equation}     \label{in:nodegconDemcoqd}
   \dist{\nullv^*}{\Demcoqd f(\X\backslash\Argmin{f})}=\tau>0.
\end{equation}
Then, $f$ admits global weak sharp minimizers and $\wshar{f}\ge\tau$.
\end{theorem}

\begin{proof}
Since $f$ is l.s.c. on $\X$, set $\X\backslash\Argmin{f}$ is open.
By exploiting the estimate established in Lemma \ref{lem:stslDemcoqdestim},
one obtains in force of condition (\ref{in:nodegconDemcoqd})
$$
   \inf_{x\in\X\backslash\Argmin{f}}\stsl{f}{x}\ge
   \inf_{x\in\X\backslash\Argmin{f}}\dist{\nullv^*}{\Demcoqd f(x)}
   =\tau>0.
$$
Thus the nondegeneracy condition (\ref{inc:stslwshacon}) of Theorem
\ref{thm:wsharsufcondstsl} appears to be fulfilled. This allows
one to achieve both the assertions in the thesis.
\end{proof}

Theorem \ref{thm:qdwsha} extends to the broader class of q.d.
functions a well-known sufficient condition for the global weak sharp
minimality of convex functions (see, for instance, \cite{Zali01}),
which is stated below.

\begin{corollary}      \label{cor:convwsha}
Let $(\X,\|\cdot\|)$ be a Banach space and let $f:\X\longrightarrow
\R\cup\{+\infty\}$ be a function l.s.c. and convex on $\X$. If
\begin{equation}     
   \dist{\nullv^*}{\partial f(\X\backslash\Argmin{f})}=\tau>0,
\end{equation}
then $f$ admits global weak sharp minimizers and $\wshar{f}\ge\tau$.
\end{corollary}

\begin{proof}
If $f$ is continuous on $\X$ it is enough to observe that, as
remarked in Example \ref{ex:convf}, under the current
assumptions it is $\Demcoqd f=\partial f$, and then to apply Theorem
\ref{thm:qdwsha}.

More in general, the Br\o ndsted-Rockafellar theorem is known to
ensure the nonemptiness of the set valued map $\partial f$ norm
densely in $\dom(f)$. In the light of Example \ref{ex:stsldifconv},
at each subdifferentiable point it is $\stsl{f}{x}=\dist{\nullv^*}{\partial f(x)}$.
In the case $x\not\in\dom(\partial f)$, by convention it is $\dist{\nullv^*}
{\partial f(x)}=+\infty$. Thus, it is possible to apply directly Theorem
\ref{thm:wsharsufcondstsl}, condition (\ref{inc:stslwshacon})
being fufilled.
\end{proof}

Theorem \ref{thm:qdwsha} is not the only extension of the sufficient condition
valid for the convex case. In fact, in \cite{NgZhe03} a similar result was
presented, which relies on a nondegeneracy condition involving
Fr\'echet subdifferentials. In contrast to Theorem \ref{thm:qdwsha},
such result requires an Asplundity assumption on the underlying
Banach space. On the other hand, in the former condition a generalized
differentiability  assumptions is made on $f$, while this can be
avoided in \cite{NgZhe03}, because nonemptiness of Fr\'echet
subdifferentials is guaranteed norm densely by virtue of the Asplund property
(see \cite{Mord06}). Thus these two conditions can not be obtained
from the other each. Notice that the key step in the proof presented
in \cite{NgZhe03} is the use of a fuzzy sum rule for Fr\'echet
subdifferential valid in force of the Fr\'echet trustworthiness
of the underlying space, the latter property being an equivalent
manifestation of the Asplund one. All of this is not needed in the
approach here exposed, since $\Demcoqd$ inherits from Clarke
subdifferential calculus the sum rule adequate to the present
circumstance (remember Remark \ref{rem:demdif}(3)).

Since in context of q.d. functions condition (\ref{in:nodegconDemcoqd})
implies the nondegeneracy condition (\ref{inc:stslwshacon}), the
former can not be expected to be a characterization of global
weak sharp minimality (remember Example \ref{ex:contexsufcond}).
Nevertheless, an interesting consequence of condition
(\ref{in:nodegconDemcoqd}), which is worth noting here, is that,
in the presence of an additional semicontinuity assumption
on the map $\Demcoqd f:\X\longrightarrow\CompConv(\X^*)$,
it prevents $f$ to be smooth at any point of 
$\front\Argmin{f}$, that is the boundary of $\Argmin{f}$,
as proved below.

\begin{proposition}     \label{pro:wsharnondif}
Let $(\X,\|\cdot\|)$ be a Banach space and let $f:\X\longrightarrow\R$
be a l.s.c. function, which is q.d. at each point of $\X$. Suppose that

(i) condition (\ref{in:nodegconDemcoqd}) holds for some $\tau>0$;

(ii) map $\Demcoqd f:\X\longrightarrow\CompConv(\X^*)$ is
norm-to-norm u.s.c. at each point of $\Argmin{f}$.

\noindent Then, $f$ fails to be G\^ateaux differentiable at each
point of $\front\Argmin{f}$.
\end{proposition}

\begin{proof}
According to Theorem \ref{thm:qdwsha}, it is $\Argmin{f}\ne\varnothing$.
If $\Argmin{f}\ne\X$ (that is $f$ is not constant), since $\Argmin{f}$
is closed there exists $\bar x\in\front\Argmin{f}\subseteq\Argmin{f}$.
Assume ab absurdo that $f$ is  G\^ateaux differentiable at $\bar x$.
As remarked in Example \ref{ex:convf}, in such event it is $\Demcoqd
f(\bar x)=\{\nabla f(\bar x)\}$. As it is $\bar x\in\Argmin{f}$, the well-known Fermat rule
implies $\Demcoqd f(\bar x)=\{\nullv^*\}$. By hypothesis (ii), corresponding
to $\tau/2$  there exists $\delta_\tau>0$ such that
$$
   \Demcoqd f(\ball{\bar x}{\delta_\tau})\subseteq\ball{\Demcoqd
   f(\bar x)}{\frac{\tau}{2}}=\frac{\tau}{2}\B^*,
$$
where clearly $\ball{A}{r}=\lev{r}{\dist{\cdot}{A}}$.
Since $\ball{\bar x}{\delta_\tau}\cap (\X\backslash\Argmin{f})\ne\varnothing$,
the last inclusion contradicts condition (\ref{in:nodegconDemcoqd}),
which is in force by (i). This completes the proof.
\end{proof}

Even though, as remarked above, condition  (\ref{in:nodegconDemcoqd})
is only sufficient for global weak sharp minimality, Proposition
\ref{pro:wsharnondif} shows that for a certain class of functions
global weak sharp minimality is incompatible with differentiability.
To quote a metaphor due to V.F. Demyanov (see \cite{Demy12}),
on the account of the nice properties enjoyed by weak sharp
minimality, it is possible to assert that `ugly ducklings' appear
in this circumstance to play the role of `beautiful swans'. 

\vskip0.5cm

Let us pass now to the study of conditions for global weak sharp
minimality in the context of constrained extremum problems.
Given functions $f:\X\longrightarrow\R$, $g:\X\longrightarrow\R$ and 
$h:\X\longrightarrow\R$, here optimization problems of the form
$$
     \min_{x\in\X} f(x) \hbox{ subject to }   g(x)\le 0,\ h(x)=0
    \leqno(\mathcal{P}_c)
$$
will be considered. For convenience-sake the feasible region of
$(\mathcal{P}_c)$ is denoted by $\Omega$, i.e. $\Omega=\lev{0}{g}\cap h^{-1}(0)$,
whereas the set of all its global solutions is indicated by $\Argmin{f,\Omega}$.
If all data of $(\mathcal{P}_c)$ are functions, which are q.d. at
each point of $\X$, while $f$ and $g$ are l.s.c., and $h$ continuous on $\X$,
then problem $(\mathcal{P}_c)$ will be referred
to as a {\em q.d. problem}. Note that, whenever $(\mathcal{P}_c)$ is a q.d.
problem, its feasible region is an example of what is called a q.d. set,
according to \cite{DemRub95}.

In the context of constrained extremum problems the notion of
global weak sharp minimality is modified as follows: a problem
 $(\mathcal{P}_c)$ is said to admit {\em constrained global weak
sharp minimizers} if $\Argmin{f,\Omega}\ne\varnothing$ and there
exists $\sigma>0$ such that
$$
    \sigma\cdot\dist{x}{\Argmin{f,\Omega}}\le f(x)-\inf_{x\in\Omega}f(x),
    \quad\forall x\in\Omega.
$$
The supremum over all values $\sigma$ satisfying the above inequality
will be called {\em modulus of constrained global weak sharpness}
of $(\mathcal{P}_c)$ and denoted by $\wshar{f,\Omega}$.

The next result provide a sufficient condition for global constrained
minimizers to be weak sharp.

\begin{theorem}     \label{thm:qdconstwsha}
With reference to a q.d. problem $(\mathcal{P}_c)$, suppose that

$(i)$ $\Argmin{f,\Omega}\ne\varnothing$;

$(ii)$ there exists $\tau>0$ such that
\begin{equation}       \label{in:hypot2}
   \inf_ {x\in\X\backslash\Omega}\ \dist{\nullv^*}{\Demcoqd [g]_+(x)
      +\Demcoqd |h|(x)}\ge\tau;
\end{equation}

$(iii)$ function $f$ is Lipschitz continuous on $\X$ with rank
$\ell_f$;

$(iv)$ for some $\lambda>\ell_f$ and $\zeta>0$ it holds
\begin{equation}     \label{in:hypot4}
   \inf_{x\in\X\backslash\Argmin{f,\Omega}}
      \dist{\nullv^*}{\Demcoqd f(x)+\lambda\tau^{-1}\left(\Demcoqd
    [g]_+(x)+\Demcoqd |h|(x)\right)}\ge\zeta.
\end{equation}

\noindent Then, the solutions of $(\mathcal{P}_c)$ are constrained
global weak sharp minimizers and $\wshar{f,\Omega}\ge\zeta$.
\end{theorem}

\begin{proof}
First of all, observe that under the hypotheses made, since it is  $\Omega
\ne\varnothing$, then by virtue of Corollary \ref{cor:stslerbo}
the following error bound holds true
\begin{equation}   \label{in:erboDemcoqd}
    \dist{x}{\Omega}\le\tau^{-1}\left([g(x)]_++|h(x)|\right),\quad
    \forall x\in\X.
\end{equation}
To prove this, notice that, being $(\mathcal{P}_c)$ q.d.,
function $[g]_++|h|$ is q.d. as well. Then, it suffices to
recall that, according to Lemma
\ref{lem:stslDemcoqdestim}, it results in
$$
    \stsl{([g]_++|h|)}{x}\ge\dist{\nullv^*}{\Demcoqd ([g]_++|h|)(x)}
    \ge\dist{\nullv^*}{\Demcoqd [g]_+(x)+\Demcoqd |h|(x)},
$$
as it is
$$
    \Demcoqd ([g]_++|h|)(x)\subseteq\Demcoqd [g]_+(x)
   +\Demcoqd |h|(x)
$$
and inequality (\ref{in:hypot2}) is in force.
In the light of the validity of error bound (\ref{in:erboDemcoqd}),
since $f$ is Lipschitz continuous by hypothesis $(iii)$, set $\Omega$
is closed and problem $(\mathcal{P}_c)$ admits global solutions,
it is possible to invoke the basic principle of exact penalization
\footnote{
For a statement of such principle the reader is referred to
\cite{FacPan03} (see Theorem 6.8.1). It is to be noted that, in
the mentioned reference, being formulated in finite dimensional
spaces, the principle is proved by using the
proximinality property enjoyed by any nonempty closed subset of
$\R^n$. Nevertheless, with a slight modification, the proof can be
rendered valid in any metric space.}.
According to it, if taking any $\ell>\ell_f$, it turns out that
$$
    \Argmin{f,\Omega}=\Argmin{f+\ell\tau^{-1}([g]_++|h|)}
    \quad\hbox{ and }\quad \inf_{x\in\Omega}f(x)=\inf_{x\in\X}
   [f(x)+\ell\tau^{-1}([g]_+(x)+|h|(x))],
$$
where it is to be remarked that the right-side solution set in
the first of the above equalities relates to an
unconstrained optimization problem. The objective function of
the latter is clearly l.s.c. and q.d. on $\X$ as well as bounded
from below (remember that $f$ is so).
Consequently, one can apply Theorem
\ref{thm:qdwsha}, after having shown that the nondegeneracy
condition (\ref{in:nodegconDemcoqd}) is fulfilled. This happens
provided that $\ell=\lambda$ as in (\ref{in:hypot4}), because
$$
    \Demcoqd\left( f+\lambda\tau^{-1}([g]_+|h|\right)(x)\subseteq
   \Demcoqd f(x)+\lambda\tau^{-1}\left(\Demcoqd
    [g]_+(x)+\Demcoqd |h|(x)\right).    
$$
Thus, it results in
$$
    \zeta\cdot\dist{x}{\Argmin{f,\Omega}}\le f(x)+\lambda\tau^{-1}
   \left([g]_+(x)+|h|(x)\right)-\inf_{x\in\Omega}f(x),\quad\forall
    x\in\X,
$$
which, for any $x\in\Omega$, gives the inequality to be proved.
\end{proof}

To complement the formulation of Theorem \ref{thm:qdconstwsha},
conditions (\ref{in:hypot2}) and (\ref{in:hypot4}) should be expressed
in terms of problem data, namely in terms of dual constructions
directly related to functions $g$ and $h$, not to $[g]_+$ and $|h|$,
whose involvement is only instrumental. This can be done by virtue
of the rich apparatus calculi, which is at disposal for q.d. functions.
Notice that, in both cases, the question consists in computing
$\Demcoqd [g]_+(x)+\Demcoqd |h|(x)$. This is carried out in the
below Remark, under an additional continuity assumption on $g$,
aimed at simplifying already involved formula\hskip-0.05cm{e}.

\begin{remark}
Given $g:\X\longrightarrow\R$ and $h:\X\longrightarrow\R$, suppose
that both are continuous and q.d. on $\X$, with
$$
    \qdif g(x)=[\qdsub g(x),-\qdsup g(x)]_\sim\quad\hbox{ and }\quad
    \qdif h(x)=[\qdsub h(x),-\qdsup h(x)]_\sim.
$$
As one expects, the basic tool for calculating $\Demcoqd [g]_+(x)
+\Demcoqd |h|(x)$ is the following well-known rule expressing
the quasidifferential of a pointwise $\max$-type function (see, for instance,
\cite{DemRub86,DemRub95}). If functions $\phi:\X\longrightarrow\R$
and $\psi:\X\longrightarrow\R$ are q.d. at $x$, with 
$$
    \qdif \phi(x)=[\qdsub \phi(x),-\qdsup \phi(x)]_\sim\quad\hbox{ and }\quad
    \qdif \psi(x)=[\qdsub \psi(x),-\qdsup \psi(x)]_\sim,
$$
and if $\phi(x)=\psi(x)$, then, defined function $\phi\vee\psi:
\X\longrightarrow\R$ as
$$
    (\psi\vee\psi)(x)=\max\{\phi(x),\,\psi(x)\},\quad x\in\X,
$$
the quasidifferential $\qdif (\phi\vee\psi)(x)=[\qdsub (\phi\vee\psi)(x),
-\qdsup (\phi\vee\psi)(x)]_\sim$ is given according to the
formula\hskip-0.05cm{e}
$$
    \qdsub (\phi\vee\psi)(x)=
    \clco\left[\left(\qdsub\phi(x)-\qdsup\psi(x)\right)
    \cup\left(\qdsub\psi(x)-\qdsup\phi(x)\right)\right]
    \quad\hbox{ and }\quad
    \qdsup (\phi\vee\psi)(x)=\qdsup\phi(x)+\qdsup\psi(x).
$$
Consequently, since it is $[g]_+=g\vee\nullv^*$, then if $x\in\X$
is such that $g(x)=0$, one obtains
\begin{equation}      \label{eq:qdifruleg}
   \qdsub[g]_+(x)=\clco\left[\qdsub g(x)\cup (-\qdsup g(x)) \right]
   \quad\hbox{ and }\quad
   \qdsup[g]_+(x)=\qdsup g(x).
\end{equation}
Analogously, since it is $|h|=h\vee(-h)$, then if $x\in\X$ is such
that $h(x)=0$, one obtains
\begin{equation}    \label{eq:qdifruleh}
   \qdsub |h|(x)=2\, \clco\left[\qdsub h(x)\cup (-\qdsup h(x)) \right]
   \quad\hbox{ and }\quad
   \qdsup  |h|(x)=\qdsup h(x)-\qdsub h(x).
\end{equation}
Now, if $x\in\X\backslash\Omega$, the violation of the equality/inequality
system defining $\Omega$ can be reduced to one
of the following seven cases, which in turn leads to a different
expression of set  $\Demcoqd [g]_+(x)+\Demcoqd |h|(x)$, on the
base of formula\hskip-0.05cm{e} (\ref{eq:qdifruleg}) and
(\ref{eq:qdifruleh}). Take into account that, because of the continuity
of $g$ and $h$, the nonnull sign of their values persists in a whole
neighbourhood of a point at which such sign is taken.

\vskip.25cm

\begin{itemize}

\item[{\bf case 1:}] $g(x)<0$ and $h(x)<0$. Since in this case locally
it is $[g]_+=\nullv^*$ and $|h|=-h$, it results in
$$
   \Demcoqd [g]_+(x)+\Demcoqd |h|(x)=\left(-\qdsup h(x)\right)
   \Demdif\qdsub h(x);
$$

\vskip.25cm

\item[{\bf case 2:}] $g(x)<0$ and $h(x)>0$. Since in this case locally
it is $[g]_+=\nullv^*$ and $|h|=h$, it results in
$$
   \Demcoqd [g]_+(x)+\Demcoqd |h|(x)=\qdsub h(x)\Demdif
   \left(-\qdsup h(x) \right);
$$

\vskip.25cm

\item[{\bf case 3:}] $g(x)=0$ and $h(x)<0$. Being in this case locally
$|h|=-h$, by recalling formula\hskip-0.05cm{e} (\ref{eq:qdifruleg})
one finds
$$
   \Demcoqd [g]_+(x)+\Demcoqd |h|(x)=\left\{
   \clco\left[\qdsub g(x)\cup (-\qdsup g(x)) \right]\Demdif 
   \left(-\qdsup g(x)\right)\right\}+\left[\left(-\qdsup h(x)\right)
   \Demdif\qdsub h(x)\right];
$$

\vskip.25cm

\item[{\bf case 4:}] $g(x)=0$ and $h(x)>0$. Being in this case locally
$|h|=h$, by recalling formula\hskip-0.05cm{e} (\ref{eq:qdifruleg})
one finds
$$
   \Demcoqd [g]_+(x)+\Demcoqd |h|(x)=\left\{
   \clco\left[\qdsub g(x)\cup (-\qdsup g(x)) \right]\Demdif 
   \left(-\qdsup g(x)\right)\right\}+\left[\qdsub h(x)\Demdif
   \left(-\qdsup h(x) \right)\right];
$$

\vskip.25cm

\item[{\bf case 5:}] $g(x)>0$ and $h(x)<0$. Being in this case locally
$[g]_+=g$ and $|h|=-h$, one obtains
$$
   \Demcoqd [g]_+(x)+\Demcoqd |h|(x)=
   \left[\qdsub g(x)\Demdif   \left(-\qdsup g(x) \right)\right]+
   \left[\left(-\qdsup h(x)\right)\Demdif\qdsub h(x)\right];
$$

\vskip.25cm

\item[{\bf case 6:}] $g(x)>0$ and $h(x)>0$. Being in this case locally
$[g]_+=g$ and $|h|=h$, one obtains
$$
   \Demcoqd [g]_+(x)+\Demcoqd |h|(x)=
   \left[\qdsub g(x)\Demdif   \left(-\qdsup g(x) \right)\right]+
   \left[\qdsub h(x)\Demdif\left(-\qdsup h(x) \right) \right];
$$

\vskip.25cm

\item[{\bf case 7:}] $g(x)>0$ and $h(x)=0$. In this case, locally
it is $[g]_+=g$. Thus, by recalling formula\hskip-0.05cm{e}
(\ref{eq:qdifruleh}), one gets
$$
   \Demcoqd [g]_+(x)+\Demcoqd |h|(x)=
   \left[\qdsub g(x)\Demdif   \left(-\qdsup g(x) \right)\right]+
   \left\{2\, \clco\left[\qdsub h(x)\cup (-\qdsup h(x)) \right]
    \Demdif\left( \qdsub h(x)-\qdsup h(x)\right)
   \right\}.
$$

\end{itemize}

Notice that, in order to calculate $\Demcoqd [g]_+(x)+\Demcoqd |h|(x)$
as it appears in condition (\ref{in:hypot4}), a further distinction is
needed, according to the fact that $x\in\Omega$ or $x\not\in\Omega$.
Since this can be carried out as done above, the details are omitted.
\end{remark}

\vskip.5cm


\subsection{A weak sharp minimality condition in terms of exhausters}     

\begin{theorem}      \label{thm:exhasutwsha}
Let $(\X,\|\cdot\|)$ be a Banach space and let $f:\X\longrightarrow\R$
be a function l.s.c. on $\X$. Suppose that $\Hadlowder{f}(x;\cdot)$
admits a lower exhauster $\Exhaustder{f}{x}$ at each point of $\X$
and
\begin{equation}    \label{in:nondegExhaustder}
    \inf_{x\in\X\backslash\Argmin{f}} \|\Exhaustder{f}{x}\|=\tau>0.
\end{equation}
Then, $f$ admits global weak sharp minimizers and $\wshar{f}\ge\tau$.
\end{theorem}

\begin{proof}
Ab absurdo, assume that $f$ has no global weak sharp minimizers.
This amounts to suppose that either $\Argmin{f}=\varnothing$ or
inequality (\ref{in:defwsharmin}) is not valid. In both cases,
according to the characterization provided by Theorem \ref{thm:wsharchalev}
there must exist $\hat x\in\X$ and $\hat\alpha>\inff$ such that $\hat x
\not\in\lev{\hat\alpha}{f}$ and
$$
    f(\hat x)<\inff+\tau\cdot\dist{\hat x}{\lev{\hat\alpha}{f}}.
$$
Since this inequality is strict, it is possible to find $\epsilon>0$ such that
$$
     \epsilon<\min\{\tau,\dist{\hat x}{\lev{\hat\alpha}{f}}\}
$$
and
$$
    f(\hat x)<\inff+(\tau-\epsilon)(\dist{\hat x}{\lev{\hat\alpha}{f}}-\epsilon).
$$
Thus, the Ekeland variational principle ensures the existence of
$\bar x\in\X$ such that
\begin{equation}     \label{in:EVP2tris}
  d(\bar x,\hat x)\le\dist{\hat x}{\lev{\hat\alpha}{f}}-\epsilon
\end{equation}
and
$$
    f(\bar x)<f(x)+(\tau-\epsilon)\|x-\bar x\|,\quad\forall x\in
    \X\backslash\{\bar x\}.
$$
As a consequence of the fact that $\bar x$ (globally) minimizes
function $f(\cdot)+(\tau-\epsilon)\|\cdot-\bar x\|$, one obtains
\begin{eqnarray*}  
 0 &\le &\Hadlowder{(f(\cdot)+(\tau-\epsilon)\|\cdot-\bar x\|)}(\bar x;v)
  \le\Hadlowder{f}(\bar x;v)+(\tau-\epsilon)\Hadupder{\|\cdot-\bar x\|}
  (\bar x;v)  \\
   &=& \Hadlowder{f}(\bar x;v)+(\tau-\epsilon)\|\cdot-\bar x\|'(\bar x;v),
   \quad\forall v\in\X
\end{eqnarray*}
(remember that $\|\cdot-\bar x\|$ is Lipschitz continuous and
directionally differentiable at $\bar x$).
By hypothesis $\Hadlowder{f}(\bar x;\cdot)$ admits a lower exhauster
$\Exhaustder{f}{\bar x}$. Thus, being
$$
    \Exhaustder{(\tau-\epsilon)\|\cdot-\bar x\|}{\bar x}=
   \{(\tau-\epsilon)\B^*\},
$$
from the last inequalities it follows
$$
    0\le \inf_{E\in\Exhaustder{f}{\bar x}}\ \suppf{v}{E}+(\tau-\epsilon)
     \suppf{v}{\B^*},\quad\forall v\in\X.
$$
The last inequality implies in the light of the Minkowski-H\"ormander
duality that for every $E\in\Exhaustder{f}{\bar x}$ it holds
$$
    0\le \suppf{v}{E}+(\tau-\epsilon)\suppf{v}{\B^*}=
     \suppf{v}{E+(\tau-\epsilon)\B^*},\quad\forall 
    v\in\X.
$$
Consequently, one finds
$$
    \nullv^*\in E+(\tau-\epsilon)\B^*,\quad\forall
   E\in\Exhaustder{f}{\bar x}.
$$
This inclusion implies that in every $E\in\Exhaustder{f}{\bar x}$
there must exist an element $v^*_E$ such that $\|v^*_E\|\le\tau-\epsilon$,
so that $\dist{\nullv^*}{E}\le\tau-\epsilon$. Thus, it is
$$
    \sup_{E\in\Exhaustder{f}{\bar x}}\dist{\nullv^*}{E}\le\tau-\epsilon,
$$
what contradicts hypothesis (\ref{in:nondegExhaustder}), if
$x\not\in\Argmin{f}$.
Therefore one is forced to deduce that $\bar x\in\Argmin{f}(\bar x)
\ne\varnothing$. To reach the concluding contradiction it is
sufficient to observe that, by virtue of inequality (\ref{in:EVP2tris}),
it is
$$
   d(\hat x,\Argmin{f})\le\dist{\hat x}{\lev{\hat\alpha}{f}}-\epsilon,
$$
whereas, being $\Argmin{f}\subseteq\lev{\hat\alpha}{f}$, it should
be at the same time
$$
   d(\hat x,\Argmin{f})\ge\dist{\hat x}{\lev{\hat\alpha}{f}}.
$$
Again the characterization provided by Theorem \ref{thm:wsharchalev}
allows to obtain the estimate of $\wshar{f}$ asserted in the thesis.
\end{proof}

\begin{remark}
Theorem \ref{thm:exhasutwsha} can be regarded as a further extension
to nonconvex functions of Corollary \ref{cor:convwsha}. In fact,
it applies to a class of functions that may not be q.d.. Concerning such
class, it is worth noting that every function defined on a normed
vector space $(\X,\|\cdot\|)$ and locally Lipschitz around a reference point
$x\in\X$ has Hadamard derivatives at $x$ admitting lower exhausters
(see, for instance, \cite{Cast00}). Again, if a Banach space $(\X,\|\cdot\|)$ is locally
uniformly convex, i.e.
$$
    \forall u\in\Sfer,\ \forall\epsilon>0\quad\hbox{it holds}\quad
    \sup_{v\in\Sfer\backslash\inte\ball{u}{\epsilon}}
    \left\|\frac{u+v}{2}\right\|<1,
$$
every function $f:\X\longrightarrow\R$, whose Hadamard lower
derivative at $x$ is continuous with respect to the argument direction,
admits such class as $\Exhaustder{f}{x}$,
as a consequence of known results on the dual representation of
p.h. functions (see, for instance, \cite{Uder00}). The class of all locally
uniformly convex Banach spaces is actually rather wide, as it
includes for example all reflexive Banach spaces.
\end{remark}

The next result provides a condition for constrained global solutions
of $(\mathcal{P}_c)$ to be weak sharp minimizers, in terms of
lower exhausters. In this case, the constraints are left in a
geometrical form, in the sense that they are not formalized by
an equality/inequality system.
In what follows, by $N^\circ(\Omega,x)$ the Clarke normal cone
to set $\Omega$ at a point $x\in\X$ is denoted. In the case $x\not
\in\Omega$ set $N^\circ(\Omega,x)=\varnothing$.

\begin{theorem}      \label{thm:exhasutconstwsha}
Let $(\X,\|\cdot\|)$ be a Banach space, let $f:\X\longrightarrow\R$
be a Lipschitz continuous function on $\X$, with rank $\ell_f>0$,
and let $\Omega\subseteq\X$ be a nonempty closed set.
Suppose that

(i) $\Argmin{f,\Omega}\ne\emptyset$;

(ii) for some $\lambda>\ell_f$ it is
\begin{equation}
   \inf_{x\in\Omega\backslash\Argmin{f,\Omega}}
   \sup_{E\in\Exhaustder{f}{x}}\dist{\nullv^*}{E+\lambda
                 (N^\circ(\Omega,x)\cap\B^*)}=\zeta>0.
\end{equation}

\noindent Then, the solutions of $(\mathcal{P}_c)$ are constrained
global weak sharp minimizers and $\wshar{f,\Omega}\ge\zeta$.
\end{theorem}

\begin{proof}
As in the constrained case already treated, let us start with
observing that, under the current assumptions, one can reduce
the constrained problem  $(\mathcal{P}_c)$ to an unconstrained one.
Indeed, since $f$ is Lipschitz continuous, with rank $\ell_f$, and
$\Argmin{f,\Omega}\ne\emptyset$, then according to the basic
penalization principle one has that, for every $\ell>\ell_f$, it holds
$$
    \Argmin{f,\Omega}=\Argmin{f+\ell\dist{\cdot}{\Omega}}.
$$
Henceforth it is possible to follow the argument proposed in the
proof of Theorem \ref{thm:exhasutwsha}, with function $f$ replaced
here by $f+\lambda\dist{\cdot}{\Omega}$. To this regard, observe that,
since function $f$ is Lipschitz continuous, its Hadamard lower
derivative admits lower (also upper) exhausters at each point
of $\X$. In the present case, since $\Argmin{f,\Omega}\ne
\emptyset$, one needs only to assume ab absurdo that inequality
(\ref{in:wshachar}) is not true. Thus, fixed $\hat x\in\X$ and
$\epsilon>0$, by proceeding as in the proof of Theorem \ref{thm:exhasutwsha}
one gets the existence of $\bar x\in\X$ such that
$$
   d(\bar x,\hat x)\le\dist{\hat x}{\lev{\hat\alpha}
   {(f+\lambda\dist{\cdot}{\Omega})}}-\epsilon
$$
and
\begin{eqnarray*}  
 0 &\le &\Hadlowder{(f(\cdot)+\lambda\dist{\cdot}{\Omega}+
        (\zeta-\epsilon)\|\cdot-\bar x\|)}(\bar x;v) \\
     &\le & \Hadlowder{f}(\bar x;v)+\Hadupder{\dist{\cdot}{\Omega}}(\bar x;v)
  +(\zeta-\epsilon)\|\cdot-\bar x\|'(\bar x;v) \\
     &\le & \Hadlowder{f}(\bar x;v)+\dist{\cdot}{\Omega}^\circ(\bar x;v)
  +(\zeta-\epsilon)\|\cdot-\bar x\|'(\bar x;v), \quad\forall v\in\X,
\end{eqnarray*}
where $\dist{\cdot}{\Omega}^\circ(\bar x;v)$ denotes the Clarke derivative
of function $\dist{\cdot}{\Omega}$ at $\bar x$ in the direction $v$. By
recalling the dual representation
$$
   \dist{\cdot}{\Omega}^\circ(\bar x;v)=
   \max_{x^*\in \partial^\circ\dist{\cdot}{\Omega}(\bar x)}\ 
   \langle x^*,v\rangle,\quad\forall v\in\X,
$$
and the fact that
$$
   \partial^\circ\dist{\cdot}{\Omega}(\bar x)=N^\circ(\Omega,\bar x)
   \cap\B^*
$$
(see, for instance, \cite{Clar83}), from the last inequalities one obtains
$$
    0\le \inf_{E\in\Exhaustder{f}{\bar x}}\ \suppf{v}{E}+
       \suppf{v}{\lambda(N^\circ(\Omega,\bar x) \cap\B^*)}+
     (\zeta-\epsilon)\suppf{v}{\B^*},\quad\forall v\in\X.
$$
Thus, for every $E\in\Exhaustder{f}{\bar x}$ one has
$$
    \nullv^*\in E+\lambda(N^\circ(\Omega,\bar x) \cap\B^*)
    + (\zeta-\epsilon)\B^*.
$$
The last inclusion implies the fact that $\bar x\in\Omega$ and
the existence of $v^*\in E+\lambda
(N^\circ(\Omega,\bar x) \cap\B^*)$, with $\|v^*\|\le\zeta-\epsilon$,
what leads to deduce that $\bar x\in\Argmin{f,\Omega}$.
This conclusion, if reasoning as in the proof of Theorem 
\ref{thm:exhasutwsha}, allows one to reach the desired
contradiction.
\end{proof}

\vskip1cm

\bibliographystyle{amsplain}

\vskip1cm

\end{document}